\documentclass[oneside,11pt]{amsart}
\usepackage{amssymb,amscd,amsmath,latexsym}
\usepackage[mathcal]{euscript}
\usepackage{hyperref}
\usepackage{slashbox}

\newcommand{\Ext}{\operatorname{Ext}}
\newcommand{\Hom}{\operatorname{Hom}}
\newcommand{\rad}{\operatorname{rad}}
\newcommand{\Ker}{\operatorname{Ker}}
\newcommand{\sgn}{\operatorname{sgn}}
\newtheorem{theorem}{Theorem}[section]
\newtheorem{corollary}[theorem]{Corollary}
\newtheorem{lemma}[theorem]{Lemma}

\theoremstyle{definition}

\newtheorem{problem}[theorem]{Problem}
\numberwithin{equation}{section}

\theoremstyle{remark}

\title[Homomorphism Spaces between Specht modules]{Large Dimension Homomorphism Spaces between Specht modules for Symmetric Groups}
\author{Craig J. Dodge}
\address{Department of Mathematics\\ University at Buffalo, SUNY \\
244 Mathematics Building\\Buffalo, NY~14260, USA}
\email{cjdodge@buffalo.edu}
\newcommand{\abs}[1]{\ensuremath{|   #1   |}}
\thanks{Thanks to my advisor Dr. David J. Hemmer whose support and suggestions were invaluable to this research.}

\begin{document}

\maketitle

\begin{abstract}
Let $F$ be a field of characteristic $p$.  We show that $\Hom_{F\Sigma_n}(S^\lambda , S^\mu)$ can have arbitrarily 
large dimension as $n$ and $p$ grow, where $S^\lambda$ and 
$S^\mu$ are Specht modules for the symmetric group $\Sigma_n$.  Similar results hold for the Weyl modules of the general 
linear group.  Every previously computed example has been at most one-dimensional, with the exception of Specht modules 
over a field of characteristic two.  The proof uses the work of 
Chuang and Tan, providing detailed information about the radical series of Weyl modules in Rouquier blocks.

\end{abstract}

\section{Introduction}

For Specht modules $S^\lambda$ and $S^\sigma$ of the symmetric group $\Sigma_n$, very little is known about the space 
$\Hom_{F\Sigma_n}(S^\lambda, S^\sigma)$. In odd characteristic this homomorphism space is isomorphic to 
$\Hom_{S_n}(\Delta(\lambda), \Delta(\sigma))$ where $\Delta(\lambda)$ is the Weyl module for the Schur algebra, $S_n$.  
In every previously known example 
for the Weyl modules, and every known symmetric group example in odd characteristic, the dimension of this space is at 
most one.   Cox and Parker were able to find an upper bound of one on the dimension 
of the homomorphism space between a 
particular class 
of Weyl modules \cite[4.4]{CoxParkerExtSchurSymm}.  They wrote, ``One might hope that all 
Hom-spaces were at most one-dimensional ... However while we know of no examples of Hom-spaces for induced modules 
which are greater than one dimensional, this may simply be because such Hom-spaces are known in certain very 
special cases.'' \cite[4.6.3]{CoxParkerExtSchurSymm}  Cox and Parker later proved that the dimensions 
of the Hom-space between the Weyl modules for $SL_3(k)$ were also bounded by one \cite[5.1]{CoxParkerWeylHom}.  Fayers 
and Martin were able to improve upon the Carter-Payne homomorphisms, 
which yielded new homomorphisms between Specht modules \cite[5.2.22]{FayersMartinSpechtHom}, but still found no 
examples of homomorphism spaces with dimension greater than one.  Murphy has shown in characteristic two 
that a Specht module can be 
decomposable with arbitrarily many summands, which implies the Hom-space from a Specht module to itself can have 
arbitrarily large dimension \cite[4.4]{MurphyDecomp}.  However this result 
implies no consequences for the Weyl modules of the general linear group.  The author would like to thank the referee for 
his or her comments and suggestions following the original submission of this paper.

\section{Preliminaries}
Let $\Sigma_n$ denote the symmetric group on $n$ elements and $F$ a field of characteristic $p$, where $p$ is an odd prime.  
Let $\Lambda$ denote the set of all partitions.  For a given 
partition $\lambda = (\lambda_1, \lambda_2, ... )$, let $\lambda'$ denote the conjugate partion and let $\lambda \vdash 
\abs{\lambda} $.

We will call $\lambda$ $p$-singular if $\lambda_{i+1} = ...= \lambda_{i+p} \neq 0$ for some $i$ and say $\lambda$ is 
$p$-regular if it is not $p$-singular.  If $\lambda'$ is $p$-regular we say $\lambda$ is $p$-restricted.

Denote $S^{\lambda}$ as the Specht module of $F \Sigma_n$ associated to the partition $\lambda$ \cite[8.4]{Jamesbook}.  
Over a field of characteristic zero the Specht modules are a complete set of non-isomorphic 
irreducible modules, but they are defined over 
any field.  
Let $S(\lambda) = (S^{\lambda})^*$ be its dual module.  
If $\lambda$ is 
$p$-restricted, let $D(\lambda)$ be the unique irreducible quotient of $S(\lambda)$.  The set 
\[
 \{D(\lambda) | \lambda \text{ is } p\text{-restricted  and } \lambda \vdash n \}
\]
forms a complete list of the irreducible modules for $F \Sigma_n$.  We will use this notation for simple modules 
to keep consistent with 
the work of Chuang and Tan \cite[2.3]{ChuaTanRoquierfilt}.  Note that $S^{\lambda} \cong S(\lambda ' ) \otimes \sgn$ where 
$\sgn$ is the sign representation.  Thus for $\lambda$ $p$-regular, $D^{\lambda} \cong D(\lambda ') \otimes \sgn$ 
where $D^{\lambda}$ is unique 
simple quotient of the Specht module $S^\lambda$.

To each partition $\lambda$ and prime $p$ we associate a $p$-core, $core(\lambda)$, and a $p$-quotient 
\[
\bar \lambda = quot(\lambda)=(\lambda^0, ..., \lambda^{p-1}) 
\]
where $\lambda^i \in \Lambda$ for each $i$. That is to say the $p$-quotient of $\lambda$ is a $p$-tuple of partitions.  
Once again we will choose notation to be consistent with Chuang and Tan \cite[2.1]{ChuaTanRoquierfilt}. 
We will choose to display $\lambda$ on a James $p$-abacus \cite[2.7]{JamesKerberbook}, with $p$ runners labeled 
$0,1,...,p-1$, such that if all the beads were pushed to the top of each runner, the first empty node would 
appear in the $0$-runner.
Then $\lambda^i$ can be read off of the $i^{th}$ runner.
Also $\lambda$ has an associated $p$-weight, $w$, 
where $w = \sum_j \abs{\lambda^j}$.  If $\lambda = core(\lambda)$, we say $\lambda$ is a $p$-core.
If $\lambda$ is a $p$-core then the abacus representing $\lambda$ will have all the 
beads pushed to the top of each runner and $\lambda^i=\varnothing$ for all $i$.

Define $\rho$ to be the $p$-core which is represented on the $p$-abacus by placing $w+i(w-1)$ beads on the $i^{th}$ runner.  
Let $\Lambda(\rho,w)$ be the set of partitions with $p$-core 
 $\rho$ and $p$-weight $w$. It is important to note the following two 
properties of partitions in $\Lambda (\rho,w)$:
\begin{enumerate}
 \item $\lambda \text{ is } p\text{-restricted} \Leftrightarrow \lambda^{p-1}= \varnothing$
 \item $\lambda \text{ is } p\text{-regular} \Leftrightarrow \lambda^{0}= \varnothing.$
\end{enumerate} 

We will let $B_w$ be the block of $F\Sigma_{\abs{\rho}+pw}$ 
associated to the $p$-core $\rho$.  The block $B_w$ is called a Rouquier block and has particularly nice properties.
The work of Chuang and Tan uses these good properties in order to compute the radical filtrations 
of certain modules in Rouquier blocks.  We will use Chuang and Tan's work to look for Specht modules within 
the block $B_w$ in order to find a large dimensional Hom-space.  The Specht modules of $B_w$ are the 
Specht modules, $S^\lambda$, such that $\lambda \in \Lambda(\rho,w)$.  For the remainder of the paper we will refer 
to the set of partitions $\Lambda(\rho,w)$ as Rouquier partitions.

\section{Radical Filtrations}

Using the notation of Chuang and Tan \cite[6.1]{ChuaTanRoquierfilt} we will define the radical series polynomial for a 
Specht module.  
For the remainder of the paper let 
$w < p$.  Let 
$\lambda , \sigma \in \Lambda( \rho , w )$ and let $\sigma$ be $p$-restricted.  Define
\[
 \rad_{S, \lambda , \sigma } (v) = \displaystyle \sum_{s \geq 0} [ \rad^s S(\lambda) / \rad^{s+1} S( \lambda ) : D ( \sigma ) ] v^s.
\]
Thus $\rad_{S, \lambda , \sigma } (v)$ is a polynomial that encodes the multiplicity of the irreducible $D(\sigma)$ in 
each radical layer of $S(\lambda)$.  Chuang and Tan proved the following formula to compute this radical polynomial using 
the $p$-quotients of these partitions and the Littlewood-Richardson coefficients. The Littlewood-Richardson coefficient 
associated to the triple of partitions $\lambda, \sigma, \mu$ is denoted
$c(\lambda;\sigma,\mu)$, with the understanding that it is 0 when 
$\abs{\lambda}\neq \abs{\sigma}+ \abs{\mu}$.

\begin{theorem} \cite[6.1]{ChuaTanRoquierfilt}
\label{ChuangTanfilt}
 Let $\lambda , \sigma \in \Lambda( \rho , w )$ and $\sigma$ be $p$-restricted.  Then 
\begin{equation}
\label{eq: ChuangTan}
 \rad_{S,\lambda, \sigma} (v) = v^{\delta(\bar \lambda , \bar \sigma)-\abs{\lambda^{p-1}}} 
\sum_{\substack{\alpha^0 , ... ,\alpha^p \in \Lambda \\  \beta^0,...,\beta^{p-1} \in \Lambda}}
\displaystyle \prod_{j=0}^{p-1} c(\lambda^j;\alpha^j,\beta^j)c(\sigma^j;\beta^j,(\alpha^{j+1})')
\end{equation}
where
\[
 \delta(\bar \lambda, \bar \sigma )= \displaystyle \sum_{j=1}^{p-1} j(\abs{\lambda^j}-\abs{\sigma^j}).
\]

\end{theorem}

Observe that this formula implies that for these Rouquier partitions, any given irreducible module 
appears in at most one radical layer, though an irreducible factor
 can appear more than once in a particular layer.
Using this formula it is possible to find Specht modules that have multiple copies of one irreducible factor.  This allows 
for the possibility of a multi-dimensional homomorphism space between the first Specht module, and the Specht module 
corresponding 
to the irreducible factor.

\section{Repeated Composition Factors}

Fix $k>0$ and define $\tau = ( k , k-1 , ... ,1 ) \vdash \frac{k^2+k}{2}$. Let $w=\frac{k^2+k}{2}+1$ and fix $p$ 
such that $w<p$.  
Define
$\gamma , \epsilon \in \Lambda( \rho , w) $, where $\rho$ is the Rouquier core associated to $w$, such that
\[
 \bar \gamma = ( [1] , [ \tau ] , \varnothing , ... , \varnothing )
\]
and 
\[
 \bar \epsilon = ( [\tau ] , [1] , \varnothing, ... , \varnothing ).
\]
We will show that $\dim \Hom_{F \Sigma_n}( S^\gamma , S^\epsilon) = k$.  For this purpose, we must first 
define $\tau(i)$ to be the partition obtained by adding one addable node to the $i^{th}$ row of $\tau$ for 
$1 \leq i \leq k+1$. Define 
$\mu(i) \in \Lambda(\rho, w)$ such that
\[
 \overline{\mu(i)} = ( [\tau(i)] , \varnothing , ... , \varnothing ).
\]

In order to compute $\rad_{S,\gamma, \epsilon} (v)$ using \eqref{eq: ChuangTan}, we need all possible choices for $\alpha$'s and $\beta$'s so that 
\begin{equation}
\label{eq: gammaepprod}
 \displaystyle \prod_{j=0}^{p-1} c(\gamma^j;\alpha^j,\beta^j)c(\epsilon^j;\beta^j,(\alpha^{j+1})') \neq 0.
\end{equation}

Figure 1 can be used to help find $\alpha$'s and $\beta$'s that give non-zero products.  
To each 
diagonal connecting $\gamma^i$ to $\epsilon^i$ we place a $\beta^i$ and to each diagonal connecting $\epsilon^{i-1}$ 
to $\gamma^{i}$ place an $\alpha^i$.  If $c(\lambda;\mu,\sigma) \neq 0$ then $\abs{\lambda}= \abs{\mu}+\abs{\sigma}$.  
Thus in order for \eqref{eq: gammaepprod} to hold, it is neccessary for the magnitude of 
the partition in the top or bottom row be equal to the sum of the magnitudes of the $\alpha$ and $\beta$ on the connecting 
diagonals.  

\begin{figure}[ht]
\begin{center}
\begin{tabular}{ccccccccccc}
 & [$\gamma^0$] & & & & [$\gamma^1$] & & & & $\varnothing$ \\
\slashbox{\tiny{[$\alpha^0$]}}{} &  & \backslashbox{}{\tiny{[$\beta^0$]}} & & \slashbox{\tiny{[$\alpha^1$]}}{} & & \backslashbox{}{\tiny{[$\beta^0$]}} & & \slashbox{\tiny{[$\alpha^2$]}}{} & & \dots   \\
& & & [$\epsilon^0$] & & & & [$\epsilon^1$]
\end{tabular}
\caption{}
\end{center}
\end{figure}

Since $\gamma^j,\epsilon^j = \varnothing$ for all $ j \geq 2$ and we need both  
$c(\gamma^j; \alpha^j, \beta^j) \neq 0$ and $c(\epsilon^j;\beta^j, (\alpha^{j+1})') \neq 0$, 
it is neccessary that
\[
 \alpha^j,\beta^j = \varnothing \text{ for all } j \geq 2.
\]
Since $$\abs{\epsilon^1}=\abs{\beta^1}+\abs{(\alpha^2)'}$$ we get $\beta^1=[1]$.  Observe next that
$$\abs{\gamma^1}= \abs{\alpha^1} + \abs{\beta^1}$$ implies $\abs{\alpha^1} = w-1$.  Similarly
$$\abs{\epsilon^0}=\abs{\beta^0}+\abs{(\alpha^1)'}$$ implies $\beta^0 = [1]$ and 
$$\abs{\gamma^0}= \abs{\alpha^0} + \abs{\beta^0}$$ implies $\alpha^0=\varnothing$.

So all that is left to determine are the possibilities for $\alpha^1$. Recall $\gamma^1=\epsilon^0=\tau$.  
If $c(\tau;\alpha^1, [1]) \neq 0$, the only choices for $\alpha^1$
are the partitions that are obtained by removing one removable node from $\tau$.  
For each choice,
\[
 c(\tau;\alpha^1, [1])=1
\]
and
\[
 c(\tau;[1],(\alpha^1)')=1.
\]
So each of these choices of $\alpha^1$ will make 
\[
 \displaystyle \prod_{j=0}^{p-1} c(\gamma^j;\alpha^j,\beta^j)c(\epsilon^j;\beta^j,(\alpha^{j+1})') =1.
\]
Since $\tau$ has $k$ removable nodes and $\delta(\bar \lambda , \bar \sigma)=w-2$ we get
\begin{equation}
\label{eq: GEradpoly}
 \rad_{S,\gamma, \epsilon} (v)=k v^{w-2}.
\end{equation}

\section{Radical Layers of $S(\gamma)$ and $S(\epsilon)$}

In order to determine if there is a homomorphism from $S(\epsilon)$ into 
$S(\gamma)$,
we will only be concerned with the radical layer of $S(\gamma)$ below 
the layer containing the composition 
factors isomorphic to $D(\epsilon)$.  First we will prove a simple lemma.

\begin{lemma}
The Loewy Length of $S(\gamma)$ is $\leq w$.
\end{lemma}
\begin{proof}
This follows from \ref{ChuangTanfilt}  since
\[
  \delta(\bar \gamma , \bar \mu) =\displaystyle \sum_{j=1}^{p-1} j(\abs{\gamma^j}-\abs{\mu^j}) \leq \sum_j j\abs{\gamma^j} = w-1
 \]
for all $\mu \in \Lambda$.
\end{proof}

Thus by \eqref{eq: GEradpoly}, the composition factors isomorphic to $D(\epsilon)$ are either on the bottom or the second from the 
bottom radical layer 
of $S(\gamma)$.  Next we will show that they are not in the bottom layer and further examine the 
radical series of both $S(\gamma)$ and $S(\epsilon)$.

It is neccessary to determine which irreducible factors can appear in a layer lower than $D(\epsilon)$. 
Notice that $\delta(\bar \lambda , \bar \mu)<w-1$ if $\abs{\mu^j} \neq 0$ for any $j \geq 1$, so in order for $D(\mu)$ 
to be in the radical layer below the irreducible factors isomorphic to $D(\epsilon)$, 
$\mu^j= \varnothing$ for all $j \geq 1$ and $\abs{\mu^0}=w$.

We need $\mu$ with $p$-quotient $\bar \mu = ( \mu^0 , \varnothing, ... ,\varnothing)$, where $\mu^0 \vdash w$, such that 
\begin{equation}
\label{Gmuprod}
 \displaystyle \prod_{j=0}^{p-1} c(\gamma^j;\alpha^j,\beta^j)c(\mu^j;\beta^j,(\alpha^{j+1})') \neq 0
\end{equation}
for some $\{\alpha^i\}_{i=0}^p$ and $\{ \beta^j \}_{j=0}^{p-1}$.

Since $\mu^j= \varnothing$ for all $j \geq 1$, we get $\alpha^j,\beta^j = \varnothing$ for all $j \geq 2$ and 
$\beta^1 = \varnothing$.  Then $$\abs{\gamma^1}=\abs{\alpha^1}+\abs{\beta^1}$$ implies $\alpha^1 = \gamma^1=\tau$.  Also
$$w=\abs{\mu^0}=\abs{\beta^0}+\abs{\alpha^1}=\abs{\beta^0}+\abs{\tau},$$ thus $\beta^0=[1].$ Finally
$$\abs{\gamma^0}=\abs{\alpha^0}+\abs{\beta^0},$$ forcing $\alpha^0=\varnothing$. So in order for \eqref{Gmuprod} hold, 
it is neccessary to find a $\mu^0$ so that $c(\mu^0;[1],(\tau)') \neq 0$.  The possible choices for 
$\mu^0$ are the partitions obtained by adding one addable node to the partition $\tau$, which are the partitions 
$\{  \tau(i) \}_{i=1}^{k+1}$.  These choices for $\mu^0$ correspond to the partitions $\mu(i)$ and for each 
$1 \leq i \leq k+1$,
\[
   \displaystyle \prod_{j=0}^{p-1} c(\gamma^j;\alpha^j,\beta^j)c(\mu(i)^j;\beta^j,(\alpha^{j+1})') =1.         
\]
Thus for each $i$ we get
\[
 \rad_{S,\gamma,\mu(i)} = 1 v^{w-1}.
\]
So for each partition $\mu(i)$ we get one irreducible factor $D_i:=D(\mu(i))$ 
in the bottom layer of $S(\gamma)$, below the layer containing the $k$ copies of 
$D(\epsilon)$.

Next we will compute the entire radical series for $S(\epsilon)$, which will be shown to only have two layers. To find the 
the entire radical filtration of $S(\epsilon)$ it will be neccessary to find all possible $\mu$ such that 
\begin{equation}
\label{epprod}
 \displaystyle \prod_{j=0}^{p-1} c(\epsilon^j;\alpha^j,\beta^j)c(\mu^j;\beta^j,(\alpha^{j+1})') \neq 0.
\end{equation}

Since $\epsilon^j = \varnothing$ for all $j \geq 2$, we get $\alpha^j, \beta^j= \varnothing$ for all $j \geq 2$, making 
$\mu^j= \varnothing$ for all $j \geq 2$. Because 
$\epsilon^1 =[1]$ and $\abs{\epsilon^1}=\abs{\alpha^1}+\abs{\beta^1}$, we may break down the problem into two cases.

Case 1: $\alpha^1=\varnothing$ and $\beta^1=[1]$, in this case $\mu^1=[1]$.  Since
$$\abs{\mu^0}=\abs{\beta^0}+\abs{\alpha^1}=\abs{\beta^0}$$
and $\abs{\epsilon^0}=\abs{\beta^0}$, we conclude $\mu^0=\beta^0=\epsilon^0$ which forces $\alpha^0=\varnothing$. 
In this case 
it is clear $\mu = \epsilon$ and the product in equation \ref{epprod} is equal to one.  
Since $\delta(\bar \epsilon, \bar \epsilon) = 0$, we know there is a lone 
copy of the irreducible $D(\epsilon)$ in the top layer of $S(\epsilon)$, 
as expected for $p$-restricted $\epsilon$. 

Case 2: $\alpha^1=[1]$ and $\beta^1=\varnothing$. The equality
$$\abs{\mu^1}=\abs{\beta^1}+\abs{\alpha^2}$$ 
forces $\mu^1 = \varnothing$ and thus $\abs{\mu^0}=w$. Similarly
$$\abs{\mu^0}=\abs{\beta^0}+\abs{\alpha^1}$$ implies $\abs{\beta^0}=w-1=\abs{\epsilon^0}$, which forces 
$\alpha^0 = \varnothing$ 
and $\beta^0=\epsilon^0=\tau$. So the only choices for $\mu^0$ to make $c(\mu^0;\tau,[1]) \neq 0$ are the partitions 
that are made by adding one addable node to $\tau$, which are partitions 
$\{  \tau(i) \}_{i=1}^{k+1}$.  These choices for $\mu^0$ 
correspond to partitions $\mu(i)$
and for each of these $\mu(i)$ 
\[
 \rad_{S,\epsilon,\mu(i)}(v) = 1 v^1.
\]
So the bottom layer of $S(\epsilon)$ has the exact same irreducible factors as those 
in the bottom radical layer of $S(\gamma)$, namely
$\{ D_i \}_{i=1}^{k+1}$. 
Therefore $S(\epsilon)$ consists of only two radical layers, a top layer consisting of a lone $D(\epsilon)$ and a bottom layer with the same 
irreducible composition factors as  
the bottom radical layer of $S(\gamma)$. The following diagram illustrates the radical layers of $S(\gamma)$ and 
$S(\epsilon)$:
\[
S(\gamma) \cong
 \begin{array}{| c  c  c  c |}
\hline
D(\gamma) & & & \\ \hline
 & \vdots & & \\ 
 & \vdots & &  \\ \hline
\displaystyle \bigoplus_{i=1}^{k} D(\epsilon) & \oplus ... & \quad & \quad  \\ \hline
 \displaystyle \bigoplus_{i=1}^{k+1} D_i & & & \\ \hline
\end{array}
\qquad 
S(\epsilon) \cong
 \begin{array}{| c  c  c  c |}
\hline
 D(\epsilon) &  & &   \\ \hline
 \displaystyle \bigoplus_{i=1}^{k+1} D_i & & & \\ \hline
\end{array}
\]
\[
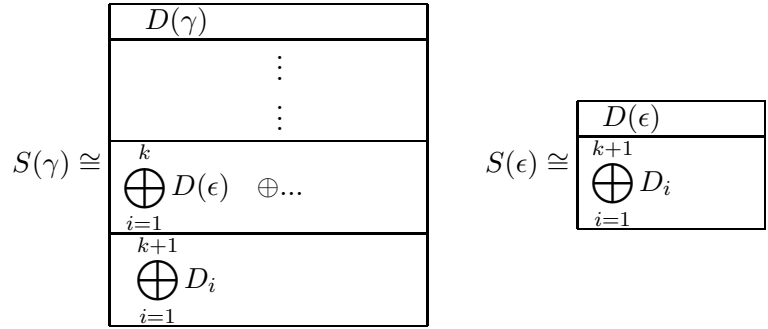
\text{Radical series of }S(\gamma),S(\epsilon).
\]
In the diagrams above we label each layer from top to bottom, $\{0,1,2,... \}$.   
The $i^{th}$ layer of the stacked boxes represents the semisimple module 
$\rad^i S(\gamma) / \rad^{i+1} S(\gamma)$ and $\rad^i S(\epsilon) / \rad^{i+1} S(\epsilon)$ respectively.

\section{Specht Module Homomorphism}

Next we will show that each copy of $D(\epsilon)$ in the second to bottom radical layer of $S(\gamma)$ will give a 
$F \Sigma_n$-homomorphism from $S(\epsilon)$ to $S(\gamma)$ that is linearly independent from the rest. To do this 
we will use another result of Chuang and Tan.

\begin{theorem} \cite[6.3]{ChuaTanRoquierfilt}  \textup{(Ext-quiver)}
\label{Extquiver}
Let $\lambda , \sigma \in \Lambda(\rho,w)$ where $\rho$ is the Rouquier Core corresponding to weight $w$ and 
$\lambda, \sigma$ are both $p$-restricted.  Then
\[
 \dim \Ext_{F \Sigma_n}^1(D(\lambda),D(\sigma)) \leq 1.
\]
\end{theorem}

Also if $P(\epsilon)$ is the projective cover of $D(\epsilon)$ and $\sigma$ is any $p$-restricted partition, we have
\[
 \dim \Ext_{F \Sigma_n}^1 ( D(\epsilon) , D(\sigma)) = \left[ \rad P(\epsilon) / \rad^2 P(\epsilon) : D(\sigma) \right] . 
\]
So there is only one copy of each of $D_i$ in the second radical layer of $P(\epsilon)$.  
Let 
$$\rad P(\epsilon) / \rad^2 P(\epsilon) \cong (\displaystyle \bigoplus_{i=1}^{k+1} D_i) \oplus ( \displaystyle \bigoplus_{j=1}^{l} V_j)$$ 
where $V_j$ are irreducible factors and $l$ is some non-negative integer.  Then $V_j \not \cong D_i$ for all $i,j$ by Theorem 
\ref{Extquiver}. 
Also we have an exact sequence
\begin{equation}
 \label{eq: exacthom}
0 \longrightarrow Q \longrightarrow P(\epsilon) \longrightarrow S(\epsilon) \longrightarrow 0. 
\end{equation}

Since $S(\epsilon)$ has radical length of two, we have 
\begin{equation}
 \label{Qexact}
0 \longrightarrow \rad^2 P(\epsilon) \longrightarrow Q \longrightarrow \displaystyle \bigoplus_{j=1}^{l} V_j \longrightarrow 0 .
\end{equation}

 Because there are $k$ copies of the irreducible factor $D(\epsilon)$ in $S(\gamma)$ we have
\[
 \dim \Hom ( P(\epsilon),S(\gamma)) = k.
\]
We want to show these homomorphisms all factor through $S(\epsilon)$, so let 
$\{ \phi_i \}_{i=1}^{k}$ be a basis for  $\Hom(P(\epsilon),S(\gamma))$.

\begin{lemma}
$Q \subseteq \Ker \phi_i$ for all $i$.
\end{lemma}
\begin{proof}
$\rad^2 P(\epsilon) \subseteq \Ker \phi_i$ since $\rad^{w-2} S(\gamma)$ has Loewy length 2, but there are no composition 
factors isomorphic to any $V_j$ in $\rad^{w-1} S(\gamma)$, so $V_j \subseteq \Ker \phi_i$ for all $i,j$ by \ref{Qexact}. 
Thus $Q \subseteq \Ker \phi_i$
\end{proof}

\begin{corollary}
 There exists a well defined map $\tilde \phi_i : S(\epsilon) \to S(\gamma)$ defined by 
\[
 \tilde \phi_i ( u + Q ) = \phi_i ( u).
\]
\end{corollary}
\begin{proof}
We can see that $\tilde \phi_i$ is induced from the exact sequence \eqref{eq: exacthom} and the map $\phi_i$.
\[
 \begin{array}{ccccccccc}
0 & \longrightarrow & Q & \longrightarrow & P(\epsilon) & \longrightarrow & S(\epsilon) & \longrightarrow & 0\\
 &      &   &   & \scriptstyle{\phi_i} \displaystyle \downarrow & \swarrow \scriptstyle  \tilde{ \phi_i} & \\
 & & & & S(\gamma)
\end{array}
\]
\end{proof}

Now it is left to show $\tilde \phi_i$ are linearly independent, so suppose 
\[
 0 = \displaystyle \sum_{i=1}^k a_i \tilde \phi_i .
\]
Then
\[
 0 = \displaystyle \sum_{i=1}^k a_i \tilde \phi_i ( u + Q) = \displaystyle \sum_{i=1}^k a_i \phi_i (u)
\]
but $u$ was arbitrary and the $\phi_i$ are linearly independent, so $a_i=0$ for all $i$. Therefore 
$\{ \tilde \phi_i \}_{i=1}^k$ 
are linearly independent.

Thus there are $k$ linearly independent homomorphisms from $S(\epsilon)$ to $S(\gamma)$,
 which implies
\[
 \dim \Hom_{F \Sigma_n} ( S(\epsilon) , S(\gamma) ) \geq k.
\]
Although if we apply $\Hom_{F \Sigma_n} (-, S(\gamma))$ to the exact sequence \eqref{eq: exacthom} we get
\[
 0 \to \Hom_{F \Sigma_n}( S(\epsilon) , S(\gamma)) \to \Hom_{F \Sigma_n}( P(\epsilon) , S(\gamma) ) .
\]
Therefore $k$ is also an upper bound on the dimension, so 

\[
 \dim \Hom_{F \Sigma_n} ( S(\epsilon) , S(\gamma) ) = k.
\]

Since $S(\lambda)=(S^\lambda)^*$, we obtain
\[
 \dim \Hom_{F \Sigma_n} ( S(\epsilon) , S(\gamma) ) = \dim \Hom_{F \Sigma_n} (S^{\gamma},S^{\epsilon}).
\]
We can conclude the following theorem.

\begin{theorem} 
 Let $\tau = ( k , k-1 , ... ,1 ) \vdash \frac{k^2+k}{2}$ and let $w=\frac{k^2+k}{2}+1<p$. Let 
$\gamma , \epsilon \in B( \rho , w) $ where $\rho$ is the Rouquier core associated to $w$, such that
\[
 \bar \gamma = ( [1] , [ \tau ] , \varnothing , ... , \varnothing )
\]
and 
\[
 \bar \epsilon = ( [\tau ] , [1] , \varnothing, ... , \varnothing )
\]
Then 
\[
 \dim \Hom_{F \Sigma_n} (S^{\gamma},S^{\epsilon}) = k.
\]

\end{theorem}

This has a direct consequence for Schur algebras. Let $S_n = S_{F}(n,n)$ be the Schur algebra over $F$ associated to 
homogeneous polynomial representations of $GL_n$ of degree $n$ \cite{GreenPolyGLn2ndedwithappendix}.  
Define $\Delta(\lambda)$ to be the Weyl Module 
associated 
to the partion $\lambda$ and $L(\lambda)$ is the unique irreducible quotient of $\Delta(\lambda)$. Carter and Lusztig 
\cite[3.7]{CarterLusztigModrepGlnSymm} 
proved that  if $\lambda$ and $\sigma$ are partitions of $n$ with $p$ odd,
\[ 
\dim \Hom_{F \Sigma_n} (S^{\lambda},S^{\sigma}) = \dim \Hom_{S_n} (\Delta(\lambda),\Delta(\sigma)). 
\]

Therefore our previous result gives following theorem.
\begin{theorem}
 Let $\tau = ( k , k-1 , ... ,1 ) \vdash \frac{k^2+k}{2}$ and let $w=\frac{k^2+k}{2}+1<p$. Let 
$\gamma , \epsilon \in B( \rho , w) $ where $\rho$ is the Rouquier core associated to $w$, such that
\[
 \bar \gamma = ( [1] , [ \tau ] , \varnothing , ... , \varnothing )
\]
and 
\[
 \bar \epsilon = ( [\tau ] , [1] , \varnothing, ... , \varnothing )
\]
Then 

\[
 \dim \Hom_{S_n} (\Delta(\gamma),\Delta(\epsilon)) = k.
\]
\end{theorem}

The problem of removing the dependence on prime $p$ is left open.

\begin{problem}
\label{OpenProb}
 Fix an odd prime $p$ and field $F$ of characteristic $p$. For arbitrary Specht modules $S^{\sigma},S^{\lambda}$, is 
\[
 \dim \Hom_{F \Sigma_n} (S^{\lambda},S^{\sigma}) 
\]
unbounded as $n$ is allowed to grow?

\end{problem}

Note that the technique used in this paper will be of no use in solving \ref{OpenProb}.  
The formulas developed by Chuang and 
Tan require $w <p$,  which for a fixed $p$ bounds $k$ by $k(k+1)/2 +1 < p$.

It is interesting to note how large the symmetric group must be to find these Specht modules.  
If $\lambda \in \Lambda(\rho,k(k+1)/2+1)$ then
\[
 \abs{\lambda}= \frac{1}{96}(p^3-p)(k^2+k)(p(k^2+k)+4) + p(k(k+1)/2+1).
\]
Which means for $k=2$ and $p=5$, which is the smallest case that we constructed, that $n=275$.  As we let 
$k$ increase, $n \approx O(k^{12})$.  This 
gives insight as to why examples of large dimensional Hom-spaces are so hard to find and 
demonstrates the power of Chuang and Tan's work to handle 
very large partitions.

\bibliography{references0210}
\bibliographystyle{plain}

\end{document}